\documentclass[10.5pt,reqno]{amsart}

\usepackage{amsbsy}
\usepackage{amscd}
\usepackage{amsfonts}
\usepackage{amsgen}
\usepackage{amsmath}
\usepackage{amsopn}
\usepackage{amssymb}
\usepackage{amstext}
\usepackage{amsthm}
\usepackage{amsxtra}
\usepackage[all]{xy}
\usepackage{comment}
\usepackage{euscript}
\usepackage{mathtools}
\mathtoolsset{showonlyrefs=true}

\usepackage{mathrsfs}
\usepackage{rotating}
\usepackage{url}
\usepackage{todonotes}

\usepackage{hyperref}
\hypersetup{
setpagesize=false,
bookmarksnumbered=true,%
bookmarksopen=true,%
colorlinks=true,%
linkcolor=blue,
citecolor=red,
}

\theoremstyle{plain}
\newtheorem{thm}{Theorem}[section]
\newtheorem{prop}[thm]{Proposition}

\newtheorem{cor}[thm]{Corollary}

\theoremstyle{definition}\newtheorem{defn}[thm]{Definition}
\newtheorem{rmk}[thm]{Remark}

\newtheorem{ques}[thm]{Question}
\newtheorem{step}{Step}

\numberwithin{equation}{section}

\renewcommand{\theta}{\vartheta}
\renewcommand{\phi}{\varphi}
\renewcommand{\epsilon}{\varepsilon}
\renewcommand{\subset}{\subseteq}

\newcommand{\N}{\mathbb N}
\newcommand{\Z}{\mathbb Z}

\newcommand{\R}{\mathbb R}
\newcommand{\C}{\mathbb C}

\newcommand{\pol}[1]{\langle #1\rangle}
\newcommand{\fps}[1]{\langle\langle #1\rangle\rangle}
\newcommand{\cf}[2]{\mr{cf}(#1; #2)}

\newcommand{\mr}[1]{\mathrm{#1}}

\newcommand{\mc}[1]{\mathcal{#1}}

\begin{document}

\title[Computing Moments of Polynomials in Semicirculars in PPtime]{Pseudo-Polynomial Time Algorithm for Computing Moments of polynomials in Free Semicircular Elements}
\author{Rei Mizuta}
\address{Graduate School of Mathematics\\University of Tokyo\\Komaba, Tokyo 153-8914, Japan}
\email{\href{mailto:}{rmizuta@ms.u-tokyo.ac.jp}}
\date{\today}

\begin{abstract}
We consider about calculating $M$th moments of a given polynomial in free independent semicircular elements in free probability theory. By a naive approach, this calculation requires exponential time with respect to $M$. We explicitly give an algorithm for calculating them in polynomial time by rearranging Sch\"utzenberger's algorithm.
\end{abstract}

\maketitle

\tableofcontents

\section{Introduction} \label{introduction}
\noindent Let $(\Omega,\mathcal{F},\mathbb{P})$ be a probability space, and $X,Y\in L^1(\Omega)$ be two independent $\R$-valued random variables whose means are zero. Cram\'er's theorem \cite{cramer} states that the sum of these two random variables follows a normal distribution if and only if both of $X$ and $Y$ are so.

On the other hand, it is known that this theorem does not have counterparts in free probability theory \cite{BV95} and there is an attempt to obtain the same result as this theorem in free probability in fixed Wigner chaos \cite{BB13}. We continue this attempt for polynomials in free independent semicircular elements, so our setting is as follows.

\begin{ques}
  Let $s_1,s_2$ be two free independent standard semicircular elements, and $p(X),q(X)$ be two polynomials of one variable such that both of them are not constant. Does $p(s_1)+q(s_2)\sim S(0,1)$ imply $p(s_1)\sim S(0,\sigma_1^2)$ and $q(s_1)\sim S(0,\sigma_2^2)$ for some $\sigma_1,\sigma_2>0$?
\end{ques}

Here $a\sim S(0,1)$ means that the spectral density of $a$ is equivalent to the semicircular density \eqref{semicircular density} for any operator $a$ and $a\sim S(0,\sigma^2)$ means $a/\sigma \sim S(0,1)$ for any positive real number $\sigma$. The above problem is generalized as follows.

\begin{ques} \label{pol_identification}
  Let $s_1,s_2,...,s_n$ be free independent standard semicircular elements and

  \noindent
  $p(X_1,X_2,...,X_n)$ be a non-commutative polynomial of $n$ variables. If $p(s_1,s_2,..,s_n)\sim S(0,1)$, is $p$ $\R$-linear?
\end{ques}

When we try to solve this problem for a given $p(X_1,X_2,...,X_n)$, the folowing strategy is available: calculating $p(s_1,s_2,...,s_n)$'s moments and comparing them with that of a standard semicircular element. Therefore the following subtask is important.

\begin{ques} \label{subtask}
  In the setting of Question \ref{pol_identification}, can we calculate $M$-th moment of

  \noindent
  $p(s_1,s_2,...,s_n)$ in practical time?
\end{ques}

While doing a naive calculation, expanding $M$-th power to $(m_p)^M$ monomials and taking summation of expectation of them, where $m_p$ is the number of monomials which appear in $p$, the computational time costs exponential time with respect to $M$. We give a practical tool for this subtask, Question \ref{subtask}, by giving an algorithm which calculates the output of Question \ref{subtask} in polynomial time with respect to $M$ by using Sch\"utzenberger's algorithm \cite{Sc62}.

In this paper, we introduce some related work about Question \ref{pol_identification} and operator algebraic research which uses Jungen's theorem in Chapter \ref{related work}. In Chapter \ref{preliminaries}, we prepare some preliminaries about free probability theory and Jungen's theorem. Finally, we show our algorithm in Chapter \ref{algorithm}.

\section{Related Work} \label{related work}

\noindent In this section, we firstly introduce some related work about Question \ref{pol_identification} in the previous chapter, secondly introduce some conventional work which uses Sch\"utzenberger's work\cite{Sc62} for operator algebras.

\subsection{Polynomial Identification Problem}

In this subsection, we introduce some conventional work which gives a partial solution for the Problem \ref{pol_identification} which is defined in the previous chapter. We mention a result in the setting of fixed Wigner chaos \cite{KNPS12}.

\begin{thm}[{\cite[Corollary 1.7]{KNPS12}}] \label{fixed chaos}
  Let $m\geq 2$ be a positive integer and $f\in L^2(\R_{+}^{m})$ be a mirror-symmetric function. Then the fourth cumulant of Wigner integral $I_m(f)$ is positive unless $f=0$ a.e.
\end{thm}

The Wigner integral $I_m(f)$ is defined in {\cite[Definition 5.3.1]{BS98}}, also mirror-symmetric is defined in {\cite[Definition 1.19]{KNPS12}} where this condition is equivalent to self-adjointness of $I_m(f)$.

Let $(T_k)_{k\geq 0}\subset \R [X]$ be the Chebyshev polynomial of second type {\cite[Chapter 5.1]{BS98}} and we define an operator $s_k$ as $s_k:=I_1(1_{[k-1,k]})$ for each positive integer $k$. Since $T_m(s_k)=I_m(1_{[k -1,k]}^{\otimes m})$ holds by an argument in the proof of {\cite[Theorem 5.3.4]{BS98}},
\begin{align}
  T_{k_1}(s_{i_i})\cdots\cdots T_{k_N}(s_{i_N})&=I_{k_1}(1_{[i_1 -1,i_1]}^{\otimes k_1})\cdots\cdots I_{k_N}(1_{[i_N -1,i_N]}^{\otimes k_N})\\
  &=I_m(1_{[i_1 -1,i_1]}^{\otimes k_1}\otimes 1_{[i_2 -1,i_2]}^{\otimes k_2} \otimes \cdots\cdots \otimes 1_{[i_N -1,i_N]}^{\otimes k_N})
\end{align}
holds for any positive integer $m$ and $k$ by the product formula of Wigner chaoses {\cite[Proposition 5.3.3]{BS98}}. Then following Corollary holds.

\begin{cor} \label{fixed chaos for pol_identification}
  Let $\C\pol{X_1,X_2,...,X_n}_{\mr{s.a.}}$ is the collection of all self-adjoint non-commutative polynomials and $m\geq 2$ be a positive integer and

  \begin{multline}
    p\in \mr{span}_{\R}\{T_{k_1}(X_{i_i})...T_{k_N}(X_{i_N})\mid N\in\N,1\leq i_1,i_2,...,i_N,k_1,k_2,...,k_N\leq n \\
    \textrm{ such that } k_1+k_2+...+k_N=m, i_j\neq i_{j+1} \textrm{ for }j=1,2,...,N-1\}\cap \C\pol{X_1,X_2,...,X_n}_{\mr{s.a}}
  \end{multline}
  be a non-commutative polynomial. Then $p(s_1,s_2,...,s_n)$ is not a standard semicircular element.
\end{cor}

However, the positivity of the fourth cumulant fails for linear combination of different chaoses. For example
\begin{align}
  \kappa_{4}(I_3(1_{[0,1]}^{\otimes 3})-2I_1(1_{[0,1]}))&=\kappa_{4}(s_1^3-3s_1)\\
  &=-2\\
  &< 0
\end{align}
  holds. So we cannot extend this argument to a general polynomial for solving Question \ref{pol_identification}.

\subsection{Sch\"utzenberger's Work in Operator Algebra}

We also remark on the work of conventional work which uses Sch\"utzenberger's work\cite{Sc62} for the region of operator algebras. In \cite{S03}, Sauer proves the rationality of Novikov-Shubin invariants in $\Z G$ if $G$ is a virtually free group. In \cite{SS15}, Shlyakhtenko and Skoufranis prove the non-atomicness of spectral distribution of polynomials in free independent semicircular elements.

We remark that these pieces of conventional research use only the existence of a proper algebraic system of a certain operator, which is defined in Definition \ref{rat_alg}, so they do not focus on the algorithm for obtaining a proper algebraic system which is suggested in \cite{Sc62}.  

\section{Preliminaries} \label{preliminaries}

\noindent We begin brief preliminaries on free probability theory and Jungen's theorem.

\subsection{Free Probability}

In this subsection, we prepare a background about free probability theory. For any von Neumann algebra $\mc{M}$, we denote the collection of all self-adjoint operators in $\mc{M}$ by $\mc{M}_{\mr{s.a.}}$.

\begin{defn} \label{def of free}
  Let $\mathcal{M}$ be a von Neumann algebra and $\tau:\mathcal{M}\rightarrow \C$ be a faithful normal tracial state. The pair $(\mathcal{M},\tau)$ is called \emph{$W^*$-probability space}.

  For any $a\in\mc{M}_{\mr{s.a.}}$, we define \emph{spectral distribution} of $a$ as the unique probability distribution $\mu$ such that
  \begin{align}
    \tau(a^m)=\int_{\R}x^md\mu(x)
  \end{align}
  for any positive integer $m$, and denote it by $\mu_{a}$.

  We call operators $x_1,x_2,...,x_n\in \mc{M}_{\mr{s.a.}}$ are \emph{free independent} if $\tau(p_1(x_{i_1})p_2(x_{i_2})...p_N(x_{i_N}))=0$ for all positive integer $N,1\leq i_1,i_2,...,i_N\leq n$ and $p_1,p_2,...,p_N\in\C[X]$ such that $\tau(p_j(x_{i_j}))=0$ for any $1\leq j\leq N$ and $i_k\neq i_{k+1}$ for all $1\leq k \leq N-1$.
\end{defn}

\begin{defn} \label{def of semicircular}
  Let $(\mathcal{M},\tau)$ be a $W^*$-probability space, an operator $s\in\mathcal{M}_{\mr{s.a.}}$ is called a \emph{standard semicircular element} if its moments are given by
  \begin{align}
    \tau(s^m)=\begin{cases}0 & m: \textrm{odd} \\ \frac{1}{m/2+1} {}_{m}C_{m/2} & m: \textrm{even}.\end{cases} \label{moments of semicircular}
  \end{align}
\end{defn}

\begin{rmk}
  In the setting of Definition \ref{def of free}, $s\in\mc{M}_{\mr{s.a.}}$ is a standard semicircular element if and only if
  
  \begin{align}
    d\mu_{s} = \frac{\sqrt{4-x^2}}{2\pi}1_{[-2,2]}(x)dx. \label{semicircular density}
  \end{align}
  
  We denote the above condition by $s\sim S(0,1)$, and we call the probability density function in right hand side of \eqref{semicircular density} the \emph{semicircular density}.
\end{rmk}

\begin{rmk} \label{uniqueness of moment}
  By above Definition \ref{def of free}, $\tau(p(x_1,x_2,...,x_n))$ is uniquely determined by the moments of $(x_i)_{1\leq i\leq n}$. In particular, by the arguments in {\cite[Chapter 1]{MS17}},
  \begin{align}
    \tau(x^M)=\sum_{\pi\in NC(M)}\kappa_{\pi}(x) \label{partition}
  \end{align}
  holds for any $x\in\mc{M}_{\mr{s.a.}}$, where $NC(M)$ means the collection of all non-crossing partition of $\{1,2,...,M\}$ {\cite[Chapter 1.8]{MS17}} and $\kappa_{\pi}$ is multiplication of cumulants of $x$ which is defined in {\cite[Chapter 2.2, Definition 8]{MS17}}.

  However, if $x$ in the left hand side of \eqref{partition} takes a polynomial in free independent operators, a summand of the right hand side of \eqref{partition} becomes  $m^M$ numbers of multiplications of cumulants for a fixed $\pi\in NC(M)$ where $m:=\#\pi$ is the number of block in $\pi$. So it takes exponential time complexity with respect to $M$ for computing \eqref{partition}, while we expand the right hand side of \eqref{partition} as multiplications of cumulants of monomial appeared in powers of $x$ and take summation of them.
\end{rmk}

\begin{rmk} \label{existness of semicirculars}
  For any positive integer $n\geq 2$, there is a $W^*$-probability space which has $n$ free independent standard semicircular elements. Let $\mathbb{F}_n$ be the free group of rank $n$. Then the free group factor $\mathcal{L}(\mathbb{F}_n)$ is defined as the weak closure of the image of the left regular representation in $B(l_2(\mathbb{F}_n))$ and has the unique faithful normal trace $\tau$. Then for each $1\leq i\leq n$, there exists a standard semicircular elements $s_i\in (\lambda_{a_i})''\subset\mathcal{L}(\mathbb{F}_n)$ ({\cite[Chapter 6]{MS17}}) and hence $s_1,s_2,...,s_n$ are free independent, where $a_1,...,a_n\in \mathbb{F}_n$ are the generators of free group and $\lambda_{g}$ is the left regular representation of $g\in \mathbb{F}_n$.
\end{rmk}

\subsection{Jungen's Theorem}

In this subsection, we give a preliminary on Sch\"utzenberger's work about Jungen's theorem \cite{Sc62}.

Let $R$ be a unital ring (possibly non-commutative), $X=\{X_1,X_2,...,X_n\}$ be a finite set and $F(X)$ be the free monoid generated by $X$. We denote the free $R$-algebra generated by $F(X)$ by $R\pol{X}$. We also denote the $R$-coefficients formal power series generated by $F(X)$ by $R\fps{X}$. We consider $R\pol{X}$ as a subring of $R\fps{X}$ by the natural inclusion. For any $F\in F(X)$ and $p\in R\fps{X}$, we denote the coefficient of $F$ in $p$ by $\cf{p}{F}$. For any $p\in R\pol{X}$ and $r_1,r_2,...,r_n\in R$, $p(r_1,r_2,...,r_n)\in R$ means the substitution of $X_1,X_2,...,X_n$ respectively for $r_1,r_2,...,r_n$.

\begin{defn} \label{rat_alg}
  We define the \emph{rational closure} $R_{\mr{rat}}\fps{X}\subset R\fps{X}$ as the smallest subring of $R\fps{X}$ which contains $\{p^{-1}\in R\fps{X}\mid p\in R\pol{X} \textrm{, which is invertible in }R\fps{X}\}$.

  We also define the \emph{algebraic closure} $R_{\textrm{alg}}\fps{X}\subset R\fps{X}$ as the all collection of $p\in R\fps{X}$ which has a \emph{proper algebraic system}, where $p$ has a proper algebraic system if
  \begin{enumerate}
  \item There are $Q_1,...,Q_L\in R\pol{X\coprod Y}$ with a finite set $Y=\{Y_1,...,Y_L\}$ such that $\cf{Q_i}{Y_j}=0$ is satisfied for all $1\leq i,j\leq L$ for some $L\in \N$. \label{existense of pas}
  \item There are $p_1,...,p_L\in R\fps{X}$ such that $p=p_1$ and $p_i=Q_i(X_1,...,X_n,p_1,...,p_L)$ are satisfied for each $1\leq i \leq L$. \label{solution of pas}
  \end{enumerate}
  
  Let $\mc{I}:R\fps{X}\rightarrow \fps{X}$ is a homomorphism of $R$-module such that it sends $p\in R\fps{X}$ to $p-\cf{p}{e}e$. We also define $R^*\pol{X},R^*\fps{X},R_{\textrm{rat}}^*\fps{X}$ and $R_{\textrm{alg}}^*\fps{X}$ respectively as the $\mc{I}$'s range of $R\pol{X},R\fps{X},R_{\textrm{rat}}\fps{X}$ and $R_{\textrm{alg}}\fps{X}$.
\end{defn}

\begin{rmk} \label{rat_star structure}
  All element $p\in R_{\textrm{rat}}^*\fps{X}$ can be obtained from $X_1,X_2,...,X_n\in R_{\textrm{rat}}^*\fps{X}$ via finite composition of following procedures {\cite{Sc62}}.
  \begin{itemize}
  \item (pseudo-inverse) $a\in R_{\textrm{rat}}^*\fps{X}\mapsto a^*:=\sum_{k=1}^{\infty}a^k$
  \item (linear combination) $r_1,r_2\in R,a,b\in R_{\textrm{rat}}^*\fps{X}\mapsto r_1a+r_2b$
  \item (multiplication) $a,b\in R_{\textrm{rat}}^*\fps{X}\mapsto ab$
  \end{itemize}

  In addition, the algebraic closure $R_{\textrm{alg}}\fps{X}$ is a subring of $R\fps{X}$ \cite{Sc62}.
\end{rmk}

\begin{rmk}
  
  so we also say $p\in R\fps{X}$ has a proper algebraic system if there are $Q_1,...,Q_L$ such that their unique solution $p_1,...,p_L$ satisfies the condition \ref{solution of pas} in Definition \ref{rat_alg}.
\end{rmk}

Next theorem is prepared for proving an analytic property of Cauchy transforms of polynomial in free independent semicircular elements \cite{SS15}.

\begin{thm}[{\cite[Lemma 5.12]{SS15}}] \label{P_semi}
  We define $P_{\mr{semi}}\in\C^*\fps{X}$ as

  \begin{align} \label{def of P_semi}
    P_{\mr{semi}}:=\sum_{F\in F(X),F\neq e}\tau(F(s_1,s_2,...,s_n))F.
  \end{align}
  Then $P_{\mr{semi}}$ is an element of $\C_{\mr{alg}}^*\fps{X}$ and whose proper algebraic system can be taken as $Q_1:=\sum_{i=1}^n(X_i(Y_1+1))^2$ with $L=1$.
\end{thm}

We remark that our definition of $P_{\mr{semi}}$ is slightly different from \cite{SS15}. We defined $P_{\mr{semi}}$ as an element of $\C^*\fps{X}$ and so their difference is only coefficients of the unit of $F(X)$.

\begin{defn} \label{hadamard}
  Let $a,b$ are elements in $R\fps{X}$, we denote the \emph{Hadamard product} of $a$ and $b$ by $a\odot b$ which is the unique element in $R\fps{X}$ defined as $\cf{a\odot b}{F}=\cf{a}{F}\cf{b}{F}$ for any $F\in F(X)$.
\end{defn}

We introduce next Jungen's theorem which are rearranged by Sch\"utzenberger in \cite{Sc62}.

\begin{thm}[{\cite[Property 2.2]{Sc62}}] \label{jungen}
  Let $R',R''\subset R$ be commuting subalgebras of $R$ and

  \noindent
  $a\in R_{\mr{rat}}^{'*}\fps{X}$,$b\in R_{\mr{alg}}^{''*}\fps{X}$ be two elements of $R^*\fps{X}$, then $a\odot b$ is an element of $R_{\mr{alg}}^*\fps{X}$.
\end{thm}

\section{Algorithm} \label{algorithm}

\noindent Let $X=\{X_1,X_2,...,X_n\}$ be a finite set, $p\in \C\pol{X_1,X_2,...,X_n}$ be a non-commutative polynomial and $M$ be a positive integer. In this chapter, we give an algorithm which calculates the $M$-th moment of $p(s_1,s_2,...,s_n)$ where $s_1,s_2,...,s_n$ are free independent standard semicircular elements.

A sketch of our algorithm is the following. Assume $p$ is an element of $\C^*\pol{X}$. Since an element 
\begin{align}
  \sum_{m\geq 1}z^mp(X_1,X_2,...,X_n)^m=(zp(X_1,X_2,...,X_n))^* \label{rat moment series}
\end{align}
is in $\C[z]_{\mr{rat}}^*\fps{X}$ by an argument in Remark \ref{rat_star structure}, we can apply Jungen's Theorem (Theorem \ref{jungen}) in previous chapter for $R=R'=\C[z], R''=\C,a=(zp(X_1,X_2,...,X_n))^*$ and $b=P_{\mr{semi}}$ which is defined in Theorem \ref{P_semi}. We then substitute each $X_1,X_2,...,X_n$ for $1$ and obtain

\noindent
$a\odot b(1,1,...,1)=:A(z)\in \C^*[[z]]$ which is well-defined. Since $\cf{A(z)}{z^m}=\tau(p(s_1,...,s_n)^m)$ for any $m\geq 1$, all we have to do is calculating $\cf{A(z)}{z^M}$, but this can be done by iterating a proper algebraic system of $a\odot b$ in $\C[z]/(z^{M+1})$ instead of $\C[z]\pol{X}$ by sending elements as
\begin{align}
  \C[z]\pol{X}\ni f\mapsto f(1,1,...,1)/(z^{M+1})\in\C[z]/(z^{M+1}). \label{sending to mod}
\end{align}
We explicitly give the procedures of the algorithm as follows.

\begin{step}[Split $p$ into $\C^*\pol{X}$ and $\C$] \label{split}
  Let $c:=p(0,0,...,0)\in \C$ be the constant part of $p$. Then we denote the reminder part by $q:=p-c\in\C^*\pol{X}$.
\end{step}

\begin{step}[Encode $(zq)^*$ as a tuple of matrices] \label{matrix}
  For obtaining a proper algebraic system of

  \noindent
  $(zq)^*\odot P_{\mr{semi}}$ by using Jungen's theorem, we encode $(zq)^*$ as a monoid homomorphism by the argument in \cite{Sc62}. Let $M_N(R)$ be the $N\times N$-matrix algebra over $R$.
  \begin{prop}[{\cite[Property 2.1]{Sc62}}] \label{monoid hom assoced with rat}
    Assume $a\in R^*\fps{X}$, the following are equivalent.
    \begin{enumerate}
    \item $a\in R_{\mr{rat}}^*\fps{X}$ \label{rationality}
    \item There are a positive integer $N\geq 2$ and a monoid homomorphism $\mu:F(X)\rightarrow M_N(R)$ such that $\cf{a}{F}=\mu(F)_{1,N}$ for any $F\in F(X)$. \label{monoid hom}
    \end{enumerate}
  \end{prop}
  We review on the constructive part of a proof in {\cite[Property 2.1]{Sc62}} for evaluating the time complexity of our algorithm.
  \begin{proof}[proof in {\cite[Property 2.1]{Sc62}}]
    Assume $a\in R_{\mr{rat}}^*\fps{X}$. All we have to do is constructing associated monoid homomorphism which satisfies \ref{monoid hom} by induction on the structure of $a$ in Remark \ref{rat_star structure}. Since $F(X)$ is a free monoid, a monoid homomorphism $\mu:F(X)\rightarrow M_N(R)$ is uniquely determined by ranges of generators $X_1,X_2,...,X_n$.
    
    If $a$ is given by $a=X_i\in R_{\mr{rat}}^*\fps{X}$ for some $1\leq i \leq n$, a monoid homomorphism $\mu$ which satisfies \ref{monoid hom} can be obtained with $N=2$ as

    \begin{align} 
      \mu (X_i) &=
      \begin{pmatrix} \label{monoid hom of var}
        0 & 1\\
        0 & 0
      \end{pmatrix}, 
      \\
      \mu (X_j) &=
      \begin{pmatrix} \label{monoid hom of var otherwise}
        0 & 0\\
        0 & 0
      \end{pmatrix} ~(i\neq j).
    \end{align}

    Then we assume there is a $a'\in R_{\mr{rat}}^*\fps{X}$ with a monoid homomorphism $\mu':F(X)\rightarrow M_{N'}(R)$ which satisfies \ref{monoid hom}. Then $a:=(a')^*$, the pseudo-inverse of $a'$, is in $R_{\mr{rat}}^*\fps{X}$, and a monoid homomorphism $\mu$ which satisfies \ref{monoid hom} can be obtained with $N=N'$ as

    \begin{align} 
      \mu (X_i) =
      \begin{pmatrix} \label{monoid hom of pseudo-inverse}
        \mu'(X_i)_{1,N} & \mu'(X_i)_{1,2} & \mu'(X_i)_{1,3} & \cdots & \mu'(X_i)_{1,N}\\
        \mu'(X_i)_{2,N} & \mu'(X_i)_{2,2} & \mu'(X_i)_{2,3} & \cdots & \mu'(X_i)_{2,N}\\
        \vdots & \vdots & \vdots & \ddots & \vdots \\
        \mu'(X_i)_{N,N} & \mu'(X_i)_{N,2} & \mu'(X_i)_{N,3} & \cdots & \mu'(X_i)_{N,N}
      \end{pmatrix} ~\textrm{for any }1\leq i\leq n.
    \end{align}

    Finally, we assume there are two rational elements $a',b'\in R_{\mr{rat}}^*\fps{X}$ with two monoid homomorphism $\mu'_1:F(X)\rightarrow M_{N_1}(R),\mu'_2:F(X)\rightarrow M_{N_2}(R)$ which satisfy \ref{monoid hom} respectively. 

    Let $r_1,r_2\in R$ be two elements. Then the linear combination $a:=r_1a'+r_2b'$ is in $R_{\mr{rat}}^*\fps{X}$ and a monoid homomorphism $\mu$ which satisfies \ref{monoid hom} can be obtained with $N=N_1+N_2+2$ as

    \begin{align} 
      \mu (X_i) =
      \begin{pmatrix}
        0 & Z^i_{1,1} & \cdots & Z^i_{1,N_1} & W^i_{1,1} & \cdots & W^i_{1,N_2} & Z^i_{1,N_1}+W^i_{1,N_2}\\
        0 & Z^i_{1,1} & \cdots & Z^i_{1,N_1} & 0 & \cdots & 0 & Z^i_{1,N_1}\\
        0 & Z^i_{2,1} & \cdots & Z^i_{2,N_1} & 0 & \cdots & 0 &Z^i_{2,N_1} \\
        \vdots & \vdots & \ddots & \vdots & \vdots & \ddots & \vdots & \vdots\\
        0 & Z^i_{N_1,1} & \cdots & Z^i_{N_1,N_1} & 0 & \cdots & 0 & Z^i_{N_1,N_1}\\
        0 & 0 & \cdots & 0 & W^i_{1,1} & \cdots & W^i_{1,N_2} & W^i_{1,N_2}\\
        0 & 0 & \cdots & 0 & W^i_{2,1} & \cdots & W^i_{2,N_2} & W^i_{2,N_2}\\
        \vdots & \vdots & \ddots & \vdots & \vdots & \ddots & \vdots & \vdots\\
        0 & 0 & \cdots & 0 & W^i_{N_2,1} & \cdots & W^i_{N_2,N_2} & W^i_{N_2,N_2}\\
        0 & 0 & \cdots & 0 & 0 & \cdots & 0 & 0
      \end{pmatrix},
    \end{align}

    for any $1\leq i\leq n$, where $Z^i:=r_1\mu'_1(X_i)$ and $W^i:=r_2\mu'_2(X_i)$.

    The multiplication $a:=a'b'$ is in $R_{\mr{rat}}^*\fps{X}$ and a monoid homomorphism $\mu$ which satisfies \ref{monoid hom} can be obtained with $N=N_1+N_2$ as

    \begin{align} 
      \mu (X_i) =
      \begin{pmatrix} 
        Z^i_{1,1} & Z^i_{1,2} & \cdots & Z^i_{1,N_1} & Z^i_{1,N_1} & 0 & \cdots & 0\\
        Z^i_{2,1} & Z^i_{2,2} & \cdots & Z^i_{2,N_1} & Z^i_{2,N_1} & 0 & \cdots & 0\\
        \vdots & \vdots & \ddots & \vdots & \vdots & \vdots & \ddots & \vdots\\
        Z^i_{N_1,1} & Z^i_{N_1,2} & \cdots & Z^i_{N_1,N_1} & Z^i_{N_1,N_1} & 0 & \cdots & 0\\
        0 & 0 & \cdots & 0 & W^i_{1,1} & W^i_{1,2} & \cdots & W^i_{1,N_2}\\
        0 & 0 & \cdots & 0 & W^i_{2,1} & W^i_{2,2} & \cdots & W^i_{2,N_2}\\
        \vdots & \vdots & \ddots & \vdots & \vdots & \vdots & \ddots & \vdots\\
        0 & 0 & \cdots & 0 & W^i_{N_2,1} & W^i_{N_2,2} & \cdots & W^i_{N_2,N_2}
      \end{pmatrix},
    \end{align}

    for any $1\leq i\leq n$, where $Z^i:=\mu'_1(X_i)$ and $W^i:=\mu'_2(X_i)$.
    
  \end{proof}

  Therefore we can obtain a monoid homomorphism associated with $(zq)^*$ by constructing that of $zq$ and take pseudo-inverse via \eqref{monoid hom of pseudo-inverse}.
  
  \begin{rmk} \label{estimation of matrix size}
    By the above construction, the size $N$ of an associated monoid homomorphism of $(zq)^*$ is estimated to be less than equal $2m_q(\deg q)+2m_q$ for given $q$ in Step \ref{split}, where $m_q$ is defined in Chapter \ref{introduction}.
  \end{rmk}
\end{step}

\begin{step}[Make a proper algebraic system of $(zq)^*\odot P_{\mr{semi}}$] \label{pas}
  We review the proof of {\cite[Property 2.2]{Sc62}} which gives a construction of a proper algebraic system of the Hadamard product in Theorem \ref{jungen}. The following theorem is a combination of {\cite[Property 2.2]{Sc62}} and {\cite[Lemma 5.12]{SS15}}.

  \begin{thm}[{\cite[Property 2.2]{Sc62}},{\cite[Lemma 5.12]{SS15}}]
    Assume $a\in R_{\mr{rat}}^*\fps{X}$ and a monoid homomorphism $\mu:F(X)\rightarrow M_N(R')$ satisfies \ref{monoid hom} in Proposition \ref{monoid hom assoced with rat}. Then $a\odot P_{\mr{semi}}\in R_{\mr{alg}}^*\fps{X}$ and its proper algebraic system can be taken with $L=N^2$ as

    \begin{align}
      (Q_1,...,Q_{N^2}):=(Q_{1,N}',Q_{\sigma_1}',...,Q_{\sigma_{N^2-1}}')
    \end{align}
    where $\sigma_1,...,\sigma_{N^2-1}$ is any permutation of $\{(i,j)\mid 1\leq i,j\leq N\}\setminus \{(1,N)\}$, and $Q'\in M_N(R\pol{X\coprod Y})$ is defined as

    \begin{align} 
    Q'=\sum_{i=1}^n(\mu(X_i)(\mc{Y}+I_{N\times N}))^2, \label{pas of zq hada P_semi}
    \end{align}

    where $\mc{Y}\in M_N(R\pol{X\coprod Y})$ is defined as $\mc{Y}_{i,j}=Y_{i+(j-1)N}$ for each $1\leq i,j\leq N$ and $I_{N\times N}$ is the multiplication unit of $M_N(R\pol{X\coprod Y})$.
  \end{thm}
  
  So we can obtain a proper algebraic system of $(zq)^*\odot P_{\mr{semi}}$ with size $L=N^2$ and it can be written as \eqref{pas of zq hada P_semi}.
  
\end{step}

\begin{step}[Iterate the proper algebraic system in Step \ref{pas} for $(\deg p)M$ times] \label{iterate}

  In this step, we calculate $\cf{(zq)^*\odot P_{\mr{semi}}(1,1,...,1)}{z^m}=\tau(q(s_1,s_2,...,s_n)^m)$ for each positive integer $m\leq M$ by using the proper algebraic system in Step \ref{pas}. Let $\pi:\C[z]\rightarrow \C[z]/(z^{M+1})$ be the quotient map.

  \begin{defn} \label{substitute by 1}
    We say $a\in \C[z]\fps{X}$ is \emph{good} if
    \begin{align}
      \{F\in F(X)\mid \cf{\cf{a}{F}}{z^{m}}\neq 0\} \label{def of good}
    \end{align}
    is finite set for any non-negative integer $m$. We denote the collection of all good elements in $\C[z]\fps{X}$ by $\C[z]_{\mr{good}}\fps{X}$.
    
    We define $\phi:\C[z]_{\mr{good}}\fps{X}\rightarrow \C[[z]]$ by defining $\phi(a)$ for each good $a$ as a unique element such that
    \begin{align}
      \cf{\phi(a)}{z^{m}}=\sum_{F\in F(X)}\cf{\cf{a}{F}}{z^{m}} \label{def of substitute}
    \end{align}
    holds for any non-negative integer $m$.
  \end{defn}

  \begin{rmk}
    All $a\in R\pol{X}\subset R\fps{X}$ are good and $a(1,1,...,1)=\phi(a)$ holds. For any good $a$ and non-negative integer $m$, summands of the right hand side of \eqref{def of substitute} are zero except for finite monomials $F$.
    
    We also remark that $(zq)^*\odot P_{\mr{semi}}\in \C[z]_{\mr{alg}}^*\fps{X}$ is good since summands of the right hand side of \eqref{def of substitute} for $a=(zq)^*\odot P_{\mr{semi}}$ are zero except for monomials $F$ such that $\cf{q^m}{F}\neq 0$. Then
    \begin{align}
      \cf{\phi((zq)^*\odot P_{\mr{semi}})}{z^m}& =\sum_{F\in F(X)}\cf{\cf{(zq)^*\odot P_{\mr{semi}}}{F}}{z^{m}}\\
      &=\sum_{F\in F(X)}\cf{\cf{(zq)^m\odot P_{\mr{semi}}}{F}}{z^{m}}\\
      &=\sum_{F\in F(X)}\cf{q^m\odot P_{\mr{semi}}}{F}\\
      &=\sum_{F\in F(X),F\neq e}\tau(F(s_1,s_2,...,s_n))\cf{q^m}{F}\\
      &=\sum_{F\in F(X)}\tau(F(s_1,s_2,...,s_n))\cf{q^m}{F}\\
      &=\tau(q(s_1,s_2,...,s_n)^m)
    \end{align}
    holds for any positive integer $m$ since $\cf{q^m}{e}=0$ holds by $\cf{q}{e}=0$.
  \end{rmk}

  Assume $a\in \C[z]_{\mr{alg}}^*\fps{X}$ and a proper algebraic system of $a$ is given as $Q_1,...,Q_L\in \C[z]\pol{X\coprod Y}$. We define $((a_1^m,...,a_L^m))_{m\geq 0}\subset (\C[z]\pol{X})^L$ as

  \begin{align}
    (a_1^0,...,a_L^0) &= (0,...,0),  \label{init pas}\\
    (a_1^{m+1},...,a_L^{m+1}) &= (Q_1(X_1,...,X_n,a_1^m,...,a_L^m),...,Q_L(X_1,...,X_n,a_1^m,...,a_L^m)) \textrm{ for }m\geq 0 . \label{inductive pas}
  \end{align}
  
  \begin{thm} \label{prf of iterate}
    Let $M$ be a positive integer and $a\in \C[z]_{\mr{alg}}^*\fps{X}$ be a good element as above. Then $\cf{\phi(a)}{z^m}=\cf{\phi(a_{1}^{M'})}{z^m}$ holds for any $m\leq M$ such that
    \begin{align}
      \max_{F\in F(X)}\{\deg F\mid 0\leq\exists m\leq M,\cf{\cf{a}{F}}{z^m}\neq 0\}\leq M'. \label{max degree}
    \end{align}

    Moreover, the left hand side of \eqref{max degree} is finite for any positive integer $M$.
  \end{thm}

  \begin{proof}
    Assume $a\in \C[z]_{\mr{alg}}^*\fps{X}$ is good and $((a_1^m,...,a_L^m))_{m\geq 0}\subset (\C[z]\pol{X})^L$ is defined as above. By an argument in \cite{Sc62}, $\cf{a}{F}=\cf{a_1^{m'}}{F}$ holds for any positive integer $m'$ and $F\in F(X)$ such that $\deg F\leq m'$. So

    \begin{align}
      \cf{\phi(a)}{z^m}&=\sum_{F\in F(X)}\cf{\cf{a}{F}}{z^m}\\
      &=\sum_{F\in F(X),\cf{\cf{a}{F}}{z^m}\neq 0}\cf{\cf{a}{F}}{z^m}\\
      &=\sum_{F\in F(X),\deg F\leq M'}\cf{\cf{a}{F}}{z^m}\\
      &=\sum_{F\in F(X),\deg F\leq M'}\cf{\cf{a_1^{M'}}{F}}{z^m}\\
      &=\cf{\phi(a_{1}^{M'})}{z^m}
    \end{align}
    holds for any non-negative integer $m\leq M$. Therefore the former statement holds.

    The latter statement holds since $\cf{\cf{a}{F}}{z^m}\neq 0$ holds only for finite monomials $F\in F(X)$ for any non-negative integer $m$ by goodness of $a$.
  \end{proof}
  
  Since the left hand side of \eqref{max degree} is less than equal $(\deg p)M$ for $a=(zq)^*\odot P_{\mr{semi}}$,
  \begin{align}
    \cf{\phi(a)}{z^m}&=\cf{\phi(a_1^{(\deg p)M})}{z^m} \label{moment by iterating}
  \end{align}
  is satisfied for any positive integer $m\leq M$ and proper algebraic system of $a$ by Theorem \ref{prf of iterate}. However, since
  \begin{align}
    a_i^{m+1}(1,1,...,1)&= Q_i(1,1,...,1,a_1^{m}(1,1,...,1),...,a_L^{m}(1,1,...,1))
  \end{align}

  holds,
  
  \begin{align}
    \pi(a_i^{m+1}(1,1,...,1))&=\pi(Q_i(1,1,...,1,a_1^{m}(1,1,...,1),...,a_L^{m}(1,1,...,1))\\
    &= \tilde{Q}_i(1,1,...,1,\pi(a_1^{m}(1,1,...,1)),...,\pi(a_L^{m}(1,1,...,1))) \label{moment by mod iterating}
  \end{align}

  holds for all $1\leq i\leq L$ and $m\geq 0$ where $\tilde{Q}_i\in \C[z]/(z^{M+1})\pol{X\coprod Y}$ is given by
  \begin{align}
    \cf{\tilde{Q}_i}{F}=\pi(\cf{Q_i}{F}).
  \end{align}
  Therefore $\cf{\phi((zq)^*\odot P_{\mr{semi}})}{z^{m}}=\cf{P_{1,N}^{(\deg p)M}}{z^{m}}$ holds for any $1\leq m\leq M$, where

  \noindent
  $(P^{m})_{m\geq 0}\subset M_N(\C[z]/(z^{M+1}))$ is defined as

  \begin{align}
    P_{j,k}^0 &= 0 \textrm{ for each }1\leq j,k\leq N, \\
    P^{m+1}&=\sum_{i=1}^n(\mu_i(P^m+\tilde{I}_{N\times N}))^2 ~\textrm{ for }m\geq 0, \label{pas of zq hada P_semi mod}
  \end{align}

  where $\mu_i\in M_N(\C[z]/(z^{M+1}))$ is defined for any $1\leq i\leq n$ as $(\mu_i)_{j,k}=\pi(\mu(X_i)_{j,k})$ for each $1\leq j,k\leq N$ and $\tilde{I}_{N\times N} \in M_N(\C[z]/(z^{M+1}))$ is the multiplication unit of $M_N(\C[z]/(z^{M+1}))$.

  So we obtain $\tau(q(s_1,s_2,...,s_n)^m)=\cf{P_{1,N}^{(\deg p)M}}{z^{m}}$ for each $1\leq m\leq M$ by iterating \eqref{pas of zq hada P_semi mod} for $(\deg p)M$ times. Since once iteration of \eqref{pas of zq hada P_semi mod} requires $n$ sumation of matrices where each summand are produced by multiplication of size $N$ matrices, this takes $O(nN^3M^2)$ time because ordinal multiplication of two polynomials $p',q'\in\C[z]/(z^{M+1})$ takes $O(M^2)$ time. Totally, the complexity of this step is $O((\deg p)nN^3M^3)$ time since we iterate \eqref{pas of zq hada P_semi} for $(\deg p)M$ times.
\end{step}

\begin{step}[Calculate the output by binomial theorem] \label{bin}
  Finally we can calculate the output

  \noindent
  $\tau(p(s_1,s_2,...,s_n)^M)$ since this is equivalent to
  \begin{align}
    \sum_{k=0}^M{}_MC_kc^k\tau(q(s_1,s_2,...,s_n)^{M-k}), \label{binomial}
  \end{align}
  while we have already calculated $(\tau(q(s_1,s_2,...,s_n)^{k}))_{1\leq k\leq M}$ in Step \ref{iterate} and $c\in\C$ is obtained in Step \ref{split}.
\end{step}

The computational bottleneck of the overall procedures is Step \ref{iterate}. Therefore we can calculate the output in $O((\deg ~p) nL^3M^3)$ time, where $L$ is the minimum size $N$ of monoid homomorphism in \ref{monoid hom} of Proposition \ref{monoid hom assoced with rat} associated with $(zq)^*\in \C[z]_{\mr{rat}}^*\fps{X}$ where $q$ is obtained in Step \ref{split}. Since $L$ has an estimation $L\leq am_p(\textrm{deg}~p)+b$ for some $a,b>0$ in Remark \ref{estimation of matrix size}, the time complexity of this algorithm is $O((\deg ~p)^4 m_p^3nM^3)$.

\section{Acknowledgements}
\noindent The author would like to thank his supervisor, Professor Yasuyuki Kawahigashi for continuing support. He also thanks Tomohiro Hayase for continuing advice about free probability theory.

\bibliographystyle{abbrv}
\bibliography{thesis}

\end{document}